\documentclass{amsart}
\usepackage[utf8x]{inputenc}
\usepackage{amsmath, amsthm, amssymb}
\usepackage[usenames,dvipsnames,svgnames,table]{xcolor}
\usepackage[margin=1.3in]{geometry}
\usepackage{esint}

\newtheorem{theorem}{Theorem}[section]
\newtheorem{proposition}[theorem]{Proposition}

\newtheorem{lemma}[theorem]{Lemma}

\newcommand{\R}{\mathbb R}

\newcommand{\eps}{\varepsilon}

\newcommand{\dd}{\, \mathrm{d}}

\DeclareMathOperator{\dist}{dist}

\numberwithin{equation}{section}

\title[Elliptic eigenvalue problems with rough coefficients]{Optimal domains for elliptic eigenvalue problems with rough coefficients}

\author{Stanley Snelson}
\address{Department of Mathematical Sciences, Florida Institute of Technology, Melbourne, FL 32901}
\email{ssnelson@fit.edu}

\author{Eduardo V. Teixeira}
\address{Department of Mathematics, University of Central Florida, Orlando, FL 32816}
\email{eduardo.teixeira@ucf.edu}

\thanks{SS was partially supported by NSF grant DMS-2213407 and a Collaboration Grant from the Simons Foundation, Award \#855061. The authors would like to thank the anonymous referee, whose detailed comments helped us improve the manuscript.}

\begin{document}

\begin{abstract}
We prove the existence of an open set minimizing the first Dirichlet eigenvalue of an elliptic operator with bounded, measurable coefficients, over all open sets of a given measure. Our proof is based on a free boundary approach: we characterize the eigenfunction on the optimal set as the minimizer of a penalized functional, and derive openness of the optimal set as a consequence of a H\"older estimate for the eigenfunction. We also prove that the optimal eigenfunction grows at most linearly from the free boundary, i.e. it is Lipschitz continuous at free boundary points.
\end{abstract}

\maketitle

\section{Introduction}


This article is concerned with the problem of minimizing the first Dirichlet eigenvalue of an elliptic operator $L = -\nabla \cdot (A(x) \nabla \cdot)$ over all open sets of a given measure. More precisely, for $d\geq 2$, suppose $A(x)$ is a symmetric, $d\times d$ matrix of bounded, measurable coefficients, with constants $0< \theta < \Theta$ such that $|A(x)| \leq \Theta$ and $A_{ij}(x) v_i v_j \geq \theta |v|^2$ for all $x, v\in \R^d$. (More succinctly, $\theta I \leq A(x) \leq \Theta I$.) For any open $\Omega\subset \R^d$, define
\begin{equation}\label{e:lambda}
\lambda_1^A(\Omega) := \inf_{\substack{u\in H_0^1(\Omega)\\u\not\equiv 0}} \frac{\int_{\Omega} \nabla u \cdot(A \nabla u) \dd x}{\int_{\Omega} u^2 \dd x}.
\end{equation}
Our goal is to find $\Omega$ minimizing $\lambda_1^A$ over the class of open sets with volume $1$. The key novelty of our problem is that no regularity assumptions are placed on $A(x)$.

In the simplest case $L = -\Delta$, the classical Faber-Krahn Theorem \cite{faber, krahn} states that for any open set $\Omega$ of finite measure, $\lambda_1(\Omega) \geq \lambda_1(B_\Omega)$, where $B_\Omega$ is a ball with the same volume as $\Omega$. This has been followed by a large (and still growing) literature on existence and regularity of optimal domains for Laplacian eigenvalues, and more generally eigenvalues of elliptic operators, subject to various constraints and boundary conditions. We give some references below in Section \ref{s:related}. 
To the best of our knowledge, every prior result on this topic either (i) imposes a regularity assumption on the coefficients $A(x)$, or (ii) works in the relaxed setting of {\it quasi-open} sets and finds a minimizer in this larger class. The present article is the first to obtain an open minimizer in the general case of bounded, measurable $A(x)$. 

Like many other authors, we use a free boundary approach that is based on a variational characterization of $u_*$, the first eigenfunction of $L$ on the optimal domain $\Omega_*$. When $A(x)$ is sufficiently regular, the Lipschitz class arises naturally as the best global regularity one can expect for $u_*$, since $|\nabla u_*|$ generally has a jump across the boundary of $\Omega_*$. 
In our case, the irregularity of the coefficient matrix $A(x)$ means that $u_*$ is merely H\"older continuous. Therefore, common free boundary techniques that are well-adapted to the Lipschitz regularity class are not applicable, and we have to replace them with arguments that are compatible with H\"older continuous state functions. 
After showing $u_*$ is H\"older and concluding $\Omega_* = \{u_*>0\}$ is open, we prove in addition that $u_*$ satisfies Lipschitz-like growth bounds at the boundary, i.e. we recover the same upper bounds at the boundary that are seen in the Laplacian case. In principle, this is counterintuitive; a solution is more regular along the free boundary then in its set of positivity.

\subsection{Main results} To state our results precisely, we first note that in general, there is nothing to prevent minimizing sequences from escaping to infinity. It is easy to construct examples where this occurs: e.g. let $B_n$ be a sequence of disjoint balls with volume 1 centered at points $x_n$ with $|x_n|\to \infty$, and let $A(x) = (1+1/n)I$ in $B_n$ and $2I$ elsewhere, where $I$ is the $d\times d$ identity matrix. 
For such $A(x)$, the infimum of $\lambda_1^A$ over open sets of volume 1 is not achieved. To prevent this situation, we could fix a large open ball $B\subset \R^d$ with $\mathcal L^d(B)\gg 1$, where $\mathcal L^d$ is Lebesgue measure, and optimize over subsets of $B$. In fact, our analysis holds true if $B\subset \R^d$ is any bounded open set with smooth boundary and volume greater than 1. For such a $B$, we define our class of admissible sets as
\[ 
\mathcal C := \{\Omega \mbox{ open}, \Omega \subset B,  \mathcal L^d(\Omega) = 1\}. 
\]
We also introduce the useful notation
\[
M_{\theta,\Theta}(B) = \{A \in L^\infty(B, \R^{d\times d}) : A(x) \text{ symmetric}, \, \theta I \leq A(x) \leq \Theta I \text{ in } B.\}.
\]
Our main result is as follows:

\begin{theorem}\label{t:main}
With $A\in M_{\theta,\Theta}(B)$, there exists a set $\Omega_* \in \mathcal C$ minimizing $\lambda_1^A(\Omega)$ over $\mathcal C$. Furthermore, the eigenfunction $u_* \in H_0^1(\Omega_*)$ corresponding to $\lambda_1^A(\Omega_*)$, extended by zero in $B\setminus \Omega_*$,  is locally H\"older continuous in $B$,  and for any compact $K\subset B$,
\[
\|u_*\|_{C^\alpha( K\cap \bar \Omega_*)} \leq C \|u_*\|_{L^2(\Omega_*)},
\]
with $\alpha\in (0,1)$ depending only on $d$, $\theta$, and $\Theta$, and $C>0$ depending on $d$, $\theta$, $\Theta$, $K$, and $\mathcal L^d(B)$.
\end{theorem}

In particular, if the optimal domain $\Omega_*$ does not intersect the boundary of $B$, then $u_*$ is globally $\alpha$-H\"older continuous on $\Omega_*$.

It is also clear from our proof that for any optimal domain (even quasi-open), the corresponding eigenfunction is entitled to the same H\"older estimate as $u_*$. In particular, any minimizer over the class of quasi-open subsets of $B$ is in fact open.

As mentioned above, one naturally cannot expect anything better than H\"older continuity for the eigenfunction $u_*$ in the interior of $\Omega_*$. However, our next result shows that $u_*$ satisfies a Lipschitz-like growth estimate along $\partial\Omega_*$:

\begin{theorem}\label{t:boundary}
With $u_*$ and $\Omega_*$ as in Theorem \ref{t:main}, and any boundary point $x_0$ of $\Omega_*$ that lies inside $B$, there are constants $C, r>0$ depending only on $d$, $\theta$, $\Theta$, $\mathcal L^d(B)$, and $\dist(x_0, \partial B)$ such that 
\[ 
\frac {|u_*(x)|} {|x-x_0|} \leq C\|u_*\|_{L^2(\Omega_*)} , \quad x\in B_r(x_0).
\]
\end{theorem}

\subsection{Related work}\label{s:related}

Summarizing the large body of work on optimal domains for eigenvalues of the Laplacian, including higher eigenvalues and functions of eigenvalues, is outside the scope of this introduction, so we refer to \cite{henrot2003second, briancon2009, bucur2012, mazzoleni2013, bucur2015lipschitz, kriventsov2018, kriventsov2019, bucur2019neumann, mazzoleni2022} and the references therein, as well as the surveys \cite{henrot2004survey, buttazzo2011review}.

Results that apply to variable-coefficient operators are less common. The well-known general result of Buttazzo-Dal Maso \cite{buttazzo1993existence} allows one to minimize $\lambda_1^A(\Omega)$ over the class of quasi-open subsets of $B$ that have volume 1. Quasi-open sets are the positivity sets of $H^1$ functions; see e.g. \cite{buttazzo1993existence} or the book \cite{bucur2005book} for more information on quasi-open sets and their role in optimization problems. The result of \cite{buttazzo1993existence} includes operators with rough coefficients and actually applies to much more general shape functionals. However, their approach does not guarantee that the minimizer $\Omega$ is an open set.



Under the assumption that $A(x)$ is Lipschitz continuous, Wagner \cite{wagner2005optimal} proved existence of a minimizing set for $\lambda_1^A$ that is open and has finite perimeter, and that the corresponding eigenfunction is Lipschitz. His proof used a penalized functional of the form
\begin{equation}\label{e:wagner}
J_\eps(w) = \frac{ \int_B \nabla w\cdot (A\nabla w) \dd x}{\int_B w^2 \dd x} + f_\eps(\mathcal L^d(\{w>0\})),
\end{equation}
where for $\eps>0$,
\[
f_\eps(t) =
\begin{cases}
 \dfrac 1 \eps(t-1),  &t \geq 1,\\[6pt]
 \eps (t-1), &t\leq 1.
 \end{cases}
 \]
Minimizers of $J_\eps$ were shown to be Lipschitz continuous using techniques inspired by the study of Alt-Caffarelli-type free boundary problems, e.g. \cite{alt1981, alt1984, aguilera1986}, and since the Lipschitz estimate is independent of $\eps$, it also applies to eigenfunctions corresponding to minimal domains for $\lambda_1^A$. 
 However, Wagner's proof uses the regularity of $A$ in an essential way, so the techniques of \cite{wagner2005optimal} 
are not applicable to our problem.


More recently, Trey \cite{trey2020lipschitz, trey2021} studied optimal domains for the first $k$ eigenvalues of an operator with H\"older continuous coefficients, and derived regularity results for both the eigenfunctions and the optimal sets. 
In \cite{russ2019}, the authors studied a shape optimization problem based on an operator with drift $-\Delta + V(x)\cdot\nabla$ for some bounded vector field $V$. We also refer to \cite{lamboley2020manifold} for the Faber-Krahn problem on a Riemannian manifold, and \cite{feldman2022} on the homogenization of eigenvalue minimizers in the context of Lipschitz continuous, periodic coefficients $A(x)$.

\subsection{Proof summary}


The basis of our argument is the following penalized functional: for any $\delta, \eps>0$, define
\begin{equation}\label{e:F-delta-eps}
\begin{split}
 F_{\delta,\eps}(u) &:= \int_B \nabla u\cdot(A\nabla u) \dd x + \dfrac 1 \delta \left|\int_B u^2\dd x - 1\right| + \dfrac 1 \eps \left(\mathcal L^d(\{u\neq 0\})-1\right)_+.
\end{split}
\end{equation}
Note that $F_{\delta,\eps}(|u|) \leq F_{\delta,\eps}(u)$, so we can always assume minimizers of $F_{\delta,\eps}$ are nonnegative. For technical reasons involving competitor functions that change sign, we use this functional rather than one defined in terms of $\mathcal L^d(\{u>0\})$. 

We prove in Lemma \ref{l:constraint} that for $\delta, \eps>0$ small enough, minimizers of $F_{\delta, \eps}$ satisfy the constraints $\int_B u^2 \dd x = 1$ and $\mathcal L^d(\{u\neq 0\})\leq 1$. The choice of $F_{\delta,\eps}$ and the proof of Lemma \ref{l:constraint} are inspired by the work of Brian\c{c}on-Hayouni-Pierre \cite{briancon2005lipschitz}, who studied optimal domains for the first Dirchlet-Laplace eigenvalue, subject to an inclusion constraint. See also \cite{friedman1995, tilli2000} for earlier, related approaches.

Next, we apply a flexible approximation method for proving H\"older continuity of minimizers of $F_{\delta,\eps}$. The key observation is that our functional $F_{\delta,\eps}$ behaves well under rescalings of the form $u \mapsto v(x) = \kappa u(x_0+rx)$. More precisely, $v$ minimizes a functional of a similar form but with $r^2/\delta$ and $\kappa^2r^2/\eps$ replacing $1/\delta$ and $1/\eps$. One may then expect that if $r^2(1+\kappa^2)$ is small enough, rescalings of $u$ are approximated by $A$-harmonic functions, and we prove via that such an approximation indeed holds, in the $L^2$ sense (Lemma \ref{l:M}). This approximation is used to prove H\"older continuity for $u$, which implies $\{u>0\}$ is open, and therefore a minimizer of the original optimization problem. 

The scaling property of $F_{\delta,\eps}$ described in the previous paragraph corresponds to Lipschitz regularity. (To see this, consider the rescaling $u\mapsto \rho^{-1} u(x_0 + \rho x)$, which preserves $\kappa^2 r^2 = 1$ as $\rho\searrow 0$.) This cannot be fully exploited in general, since the regularity of $u$ is limited by the irregularity of $A(x)$. However, at boundary points of $\{u>0\}$, we can improve our approximation argument for rescaled minimizers using our H\"older estimate, and conclude the rescaled function $v$ is uniformly close to a much smoother configuration (Lemma \ref{l:Nzero}). This uniform estimate is iterated at small scales to obtain the Lipschitz growth estimate of Theorem \ref{t:boundary}.

This strategy for proving regularity is related to earlier work of the second named author: see \cite{teixeira2013sharp} which dealt with critical Poisson equations, as well as \cite{LST2020} (joint with Lamboley and Sire) on Alt-Caffarelli free boundary problems with singular weights.


Let us comment on the dependence of constants in Theorems \ref{t:main} and \ref{t:boundary} on the measure of $B$. This dependence is technical (we believe) and arises from the estimates needed to show minimizers of our penalized functional satisfy the volume constraint, see Lemma \ref{l:constraint}(b). If $B$ is a sufficiently large ball, one might expect to prove that minimizers have diameter bounded independently of $B$, following an approach like \cite{mazzoleni2013}, which found that domains of large diameter cannot be minimizers to a spectral problem involving Laplacian eigenvalues. We do not explore this issue further in the present article.

\subsection{Open questions about the optimal domain} 
One would like to understand how regular the boundary of the optimal domain $\Omega_*$ must be. Because we are working in a rough medium, it is not obvious how much regularity we can expect. In the positive direction, the next step would be density estimates of the form $c\mathcal L^d(B_r(x_0)) < \mathcal L^d(\Omega_* \cap B_r(x_0)) < (1-c)\mathcal L^d(B_r(x_0))$ for some constant $c\in (0,1)$, where $x_0\in \partial \Omega_*$ and $r>0$ is sufficiently small. If we could prove  nondegeneracy estimates that say $u_*$ grows at least at some linear rate from $\partial \Omega_*$, then with the matching upper bounds of our Theorem \ref{t:boundary}, we could derive  density estimates for $\partial \Omega_*$ by adapting standard proofs for Alt-Caffarelli minimizers from e.g. \cite{alt1984}.

Our functional $F_{\delta,\eps}$ is not convenient for proving nondegeneracy because it is not a strictly monotonic function of $\mathcal L^d(\{u\neq 0\})$. 
A natural idea from the literature would be to use a different penalization term that rewards volumes less than one by a small amount, such as $f_\eps(\mathcal L^d(\{u>0\}))$ in \eqref{e:wagner} as in \cite{wagner2005optimal}. A related approach can be found in  \cite{briancon2009} where, working in the Laplacian case, the authors show $u_*$ is a local minimizer of a functional like
\[
\int_B |\nabla w|^2 - \lambda_* \int_B w^2  + \mu \mathcal L^d(\{w>0\}),
\]
with respect to perturbations $v$ such that $1- h \leq \mathcal L^d(\{v>0\}) \leq 1$, where $\lambda_* = \lambda_1(\Omega_*)$ and $\mu>0$ is sufficiently small, depending on $h$. 
The proof of this minimization property in \cite{briancon2009} uses the Lipschitz continuity of $u_*$ in the interior of $\Omega$, and therefore, like \cite{wagner2005optimal}, it is not directly applicable in our context.

Whether this obstacle is an artifact of the available methods, or whether nice regularity properties for $\partial\Omega_*$ are actually false in the context of discontinuous media, remains an interesting question for future work.

\subsection{Organization of the paper} In Section \ref{s:pen}, we show that minimizers of $F_{\delta,\eps}$ satisfy the $L^2$ and volume constraints for small $\delta$ and $\eps$. In Section \ref{s:holder} we prove that minimizers $u$ are H\"older continuous and conclude Theorem \ref{t:main}. In Section \ref{s:linear}, we prove the growth estimate of Theorem \ref{t:boundary}.

\section{The penalized functional}\label{s:pen}

The existence of a minimizer $u$ of $F_{\delta,\eps}$ over $H_{0}^1(B)$ follows by standard arguments. In detail, taking a minimizing sequence $u_n$, one extracts a subsequence converging weakly in $H^1(B)$, strongly in $L^2(B)$, and pointwise a.e. to a limit $u$ that is nonnegative a.e. It is well known that the first two terms in $F_{\delta,\eps}$ are lower semicontinuous with respect to weak $H^1(B)$ convergence. For the last term, we have $1_{\{u\neq 0\}} \leq \lim_{n\to \infty} 1_{\{u_n\neq 0\}}$ a.e., and Fatou's Lemma implies $(\mathcal L^d(\{u\neq 0\}) - 1)_+ \leq \lim_{n\to\infty} (\mathcal L^d(\{u_n\neq 0\}) - 1)_+$. We conclude $u$ is a minimizer for $F_{\delta,\eps}$. As mentioned above, we may assume $u\geq 0$.


Next, comparing $u$ to some fixed function $v\in H_{0}^1(B)$ with $\int_B v^2 \dd x = 1$ and $\mathcal L^d(\{v\neq 0\}) = 1$, we see that $F_{\delta,\eps}(u) \leq C_0$, or equivalently,
\begin{equation}\label{e:eps-independent}
\int_{B} \nabla u\cdot(A(x)\nabla u) \dd x + \frac 1 \delta \left|\int_B u^2 \dd x - 1\right| + \frac 1 \eps\left(\mathcal L^d(\{u\neq 0\}) - 1\right)_+ \leq C_0,
\end{equation}
for some $C_0>0$ independent of $\eps$ and $\delta$. 

We now show that minimizers of $F_{\delta,\eps}$ satisfy the $L^2$ and volume constraints that we want our optimal eigenfunction to satisfy:

\begin{lemma}\label{l:constraint}
\begin{enumerate}
\item[(a)] Let $C_0$ be the constant from \eqref{e:eps-independent}. If $\delta < \delta_0 := 1/(2C_0)$, then  any minimizer $u$ of $F_{\delta, \eps}$ over $H_{0}^1(B)$ satisfies 
\[
\int_B u^2 \dd x = 1.
\]
\item[(b)] Assume $\delta< 1/(2C_0)$ as in \textup{(a)}. There exists $\eps_0>0$ depending on $d$, $\theta$, and $\mathcal L^d(B)$, such that any nonnegative minimizer $u$ of $F_{\delta,\eps}$ over $H_{0}^1(B)$ with $\eps < \eps_0$ satisfies
\[ 
\mathcal L^d(\{u>0\}) \leq 1.
\]
\end{enumerate}
\end{lemma}

\begin{proof}
From \eqref{e:eps-independent} and our assumption $\delta< 1/(2C_0)$, we conclude $\int_B u^2 \dd x \geq \frac 1 2$, and in particular $u\not\equiv 0$. Therefore, since $F_{\delta,\eps}(u) \leq F_{\delta,\eps} (u/\|u\|_{L^2(B)})$, 
\[
\int_B |\nabla u|^2  \dd x+ \frac 1 \delta \left|\int_B u^2 \dd x - 1\right| \leq \frac {\int_B |\nabla u|^2 \dd x} {\int_B u^2 \dd x},
\]
which implies 
\begin{equation}\label{e:u-eps1}
\int_B |\nabla u|^2 \dd x \left( \int_B u^2 \dd x - 1\right) + \frac 1 \delta \left|\int_B u^2 \dd x - 1\right| \int_B u^2 \dd x \leq 0,
\end{equation}
and therefore $\int_B u^2 \dd x \leq 1$. Using this in \eqref{e:u-eps1}, we have
\[ 
\left( \int_B u^2 \dd x -1 \right)\left(\int_B |\nabla u|^2\dd x- \frac 1 \delta \int_B u^2 \dd x\right)\leq 0.
\]
If $\int_B u^2 \dd x < 1$, then this implies
\begin{equation}\label{e.one-over-delta}
\frac{\int_B |\nabla u|^2 \dd x}{\int_B u^2 \dd x} \geq \frac 1 \delta,
\end{equation}
which is a contradiction for $\delta$ small enough. Indeed, \eqref{e:eps-independent} implies $\int_B |\nabla u|^2 \dd x \leq C_0$ and $\int_B u^2 \dd x \geq 1 - C_0\delta$, so \eqref{e.one-over-delta} cannot happen if $\delta < 1/(2C_0)$.  This concludes the proof of (a).

To prove (b), following \cite{briancon2005lipschitz}, assume by contradiction that $\mathcal L^d(\{u>0\}) > 1$. Then for $t>0$ small enough depending on $u$, the function 
\[ u^t = (u - t)_+\]
also satisfies $\mathcal L^d(\{u_t>0\}) > 1$, and the minimizing property implies, using $\int u^2\dd x = 1$,
\[
\begin{split}
\int_{B} [\nabla u\cdot (A\nabla u) - \nabla u^t \cdot (A\nabla u^t)] \dd x  &\leq 
 \frac 1 \eps\left[ \left(\mathcal L^d(\{u^t>0\}) - 1\right) - \left(\mathcal L^d(\{u>0\}) - 1\right)\right]\\
 &\quad + \frac 1 \delta \left| \int_B (u^t)^2 \dd x - 1\right|\\
 &= - \frac 1 \eps \mathcal L^d(\{0<u< t\}) + \frac 1 \delta\left| \int_B (u^t)^2 \dd x - 1\right|
\end{split}
\]
and, using $\int u^2\dd x = 1$ again,
\[
\begin{split}
\int_{\{ 0< u< t\}} \nabla u \cdot (A\nabla u) \dd x +  \frac 1 \eps \mathcal L^d(\{0<u< t\}) &\leq  \frac 1 \delta\left| \int_B [u^2 -  (u^t)^2] \dd x \right|\\
&\leq \frac 1 \delta \left( \int_{\{0< u< t\}} u^2 \dd x + 2t \int_B u \dd x\right)\\
&\leq \frac 1 \delta \left( \int_{\{0< u< t\}} u^2 \dd x + 2t \mathcal L^d(B)^{1/2}\right),
\end{split}
\]
using H\"older's inequality and $\int_{B} u^2 \dd x = 1$ in the last inequality. This implies
\[
\int_{\{0< u < t\}} \left[ \theta|\nabla u|^2  + \frac 1 \eps  - \frac  {u^2} \delta \right] \dd x \leq \frac {2t \mathcal L^d(B)^{1/2}} \delta.
\]
Applying the coarea formula for Sobolev functions \cite{federer} on the left, we have
\[
\int_0^t \int_{\{u = s\}} \left( \theta |\nabla u| + \frac {1/\eps-s^2/\delta} { |\nabla u|} \right) \dd \mathcal H_{d-1} \dd s \leq \frac {2t \mathcal L^d(B)^{1/2}} \delta.
\]
Minimizing the expression $g(z) = \theta z+ (1/\eps-s^2/\delta)/ z$ over $z\in (0,\infty)$, we obtain 
\[
g(z) \geq 2[\theta(1/\eps-s^2/\delta)]^{1/2},
\]
and if $s\leq t\leq \sqrt {\delta/(2\eps)}$, we have $g(z)\geq \sqrt{2\theta/\eps}$.  Therefore, for such $t$, 
\[
\sqrt{\frac {2\theta} \eps} \int_0^t \int_{\{u =s\}} \dd \mathcal H_{d-1} \dd s \leq \frac {2t \mathcal L^d(B)^{1/2}} \delta.
\]
Applying the isoperimetric inequality\footnote{Here, we use the standard fact that almost every level set of $u\in H^1(B)$ is countably $(n-1)$ rectifiable, in order to apply the isoperimetric inequality.} to the left-hand side, this implies
\[
\sqrt{\theta}\int_0^t C_d \mathcal L^d(\{u \geq s\})^{(d-1)/d} \dd s \leq  \frac {\sqrt {2 \eps} t \mathcal L^d(B)^{1/2}} \delta.
\]
Dividing by $t$ and taking the limit as $t\to 0$, we obtain
\[ 
\sqrt{\theta}C_d  \leq \frac{\sqrt{2\eps} \mathcal L^d(B)^{1/2}} \delta,
\]
since $\mathcal L^d(\{u\geq 0\})>1$. This is a contradiction for $\eps$ sufficiently small.
\end{proof}

From Lemma \ref{l:constraint}, we immediately obtain the following important fact:

\begin{proposition}\label{p:equiv}
Assume that $\delta$ and $\eps$ are small enough as in Lemma \ref{l:constraint}.

If $u$ is a nonnegative minimizer of $F_{\delta,\eps}$ over $H_{0}^1(B)$, then $\Omega_u = \{u>0\}$ is a minimizer of $\lambda_1^A$ over 
\[
\mathcal C' := \{\Omega \text{ open}, \Omega \subset B, \mathcal L^d(\Omega) \leq 1\}.
\]
Conversely, if $\Omega$ is a minimizer of $\lambda_1^A$ over $\mathcal C'$, then the $L^2$-normalized first eigenfunction of $L$ on $\Omega$ is a minimizer of $F_{\delta,\eps}$ over $H_{0}^1(B)$. 
\end{proposition}

\section{H\"older continuity}\label{s:holder}

In this section, we prove H\"older continuity of minimizers. Throughout this section, $u$ is a minimizer of $F_{\delta,\eps}$ over $H_{0}^1(B)$, for some $\delta,\eps>0$.  For now, we allow our constants to depend on $\delta$ and $\eps$, and we do not assume that $\delta$ and $\eps$ are small as in Lemma \ref{l:constraint}.

First, we record how the minimization problem transforms under zooming in near a point:

\begin{lemma}\label{l:rescale}
For any $x_0\in B$, $\kappa>0$, $\xi\in \R$, and $0< r < \dist(x_0,\partial B)$, let
\[
v(x) := \kappa u(x_0+rx) - \xi, \quad x\in B_1.
\]
Then $v$ is a minimizer of 
\[
\begin{split}
F_{\kappa,\xi,r}(w) &= \int_{B_1} \nabla w \cdot (A(x_0+rx) \nabla w) \dd x + \frac {r^2} \delta \left|\int_{B_1} (w+\xi)^2 \dd x - \nu\right| + \frac {r^2\kappa^2} \eps \left( \mathcal L^d(\{w\neq-\xi\}) - \gamma\right)_+,\\
\nu &= \kappa^2 r^{-d}\left(1 - \int_{B \setminus B_r(x_0)} u^2 \dd x\right),\\
\gamma &= r^{-d} \mathcal L^d(\{u\neq 0\}\setminus B_r(x_0)),
\end{split}
\]
over
\[
H^1_{v}(B_1) = \{ w \in H^1(B_1) : w - v \in H_0^1(B_1)\}.
\]
\end{lemma}
\begin{proof}
Direct calculation.
\end{proof}

Next, we have a Caccioppoli estimate for rescalings of $u$, which will be needed in our compactness argument:

\begin{lemma}\label{l:caccioppoli}
With $x_0$, $\kappa$, $r$, $\xi$, and $v$ as in Lemma \ref{l:rescale}, we have
\begin{equation}\label{e:caccioppoli}
 \int_{B_{1/2}} |\nabla v|^2 \dd x \leq C \left( (\kappa^2 +\xi^2)r^2 +  \left(1+r^2\right) \int_{B_1} v^2 \dd x\right),
 \end{equation}
with $C$ depending only on $d, \theta$, $\Theta$, $\delta$, and $\eps$.
\end{lemma}

\begin{proof}
Let $\tilde A(x) = A(x_0+rx)$. Let $\zeta$ be a smooth function with compact support in $B_1$, with $0\leq  \zeta \leq 1$, to be chosen later.   Letting $w = v(1-\zeta^2)$, since $w = v$ on $\partial B_1$, we have $F_{\kappa,\xi,r}(v)\leq F_{\kappa,\xi,r}(w)$, which implies
\[ 
\begin{split} 
\int_{B_1} [\nabla v\cdot (\tilde A \nabla  v) - \nabla w\cdot (\tilde A\nabla w)] \dd x &\leq \frac{r^2} \delta \left(\left| \int_{B_1} (w+\xi)^2\dd x - \nu\right| - \left| \int_{B_1} (v+\xi)^2 \dd x - \nu\right| \right)\\
&\quad + \frac {\kappa^2r^2} \eps \left(\left( \mathcal L^d(\{w\neq -\xi\}) - \gamma\right)_+ - \left(\mathcal L^d(\{v\neq -\xi\}) - \gamma\right)_+\right)\\
&\leq \frac {r^2} \delta \left| \int_{B_1} [(w+\xi)^2 - (v+\xi)^2]\dd x \right| \\
&\quad + \frac{\kappa^2r^2} \eps \left( \mathcal L^d(\{w\neq -\xi\})  - \mathcal L^d(\{v\neq -\xi\}) \right)_+.
\end{split}
\]
Straightforward calculations with $w = v(1-\zeta^2)$ now imply
\[
\begin{split}
\int_{B_1} \nabla v\cdot (\tilde A\nabla v) \zeta^2 (2-\zeta^2) \dd x &\leq  4 \int_{B_1} \left[v^2 \zeta^2 \nabla \zeta\cdot(\tilde A \nabla \zeta) - v \zeta (1-\zeta^2)\nabla \zeta\cdot (\tilde A\nabla v)\right] \dd x \\
&\quad + \frac{r^2}\delta \left|\int_{B_1} \zeta^2 [v^2  (\zeta^2 - 2) - 2\xi v] \dd x\right|\\
&\quad  + \frac{\kappa^2r^2}\eps\left( \mathcal L^d(\{w\neq -\xi\})  - \mathcal L^d(\{v\neq -\xi\}) \right)_+.
\end{split}
\]
By the ellipticity of $\tilde A(x)$, Young's inequality, and $0\leq \zeta\leq 1$, we have
\[\begin{split}
& \theta \int_{B_1} \zeta^2 |\nabla v|^2 \dd x\\
 &\leq 4\Theta \int_{B_1} \left( v^2 |\nabla \zeta|^2 + |v| \zeta |\nabla \zeta| |\nabla v|\right)\dd x  + \frac{2r^2} \delta \left(\int_{B_1} v^2 \dd x + \xi \int_{B_1} |v| \dd x\right) +  \frac {2\kappa^2 r^2} \eps \mathcal L^d(B_1) \\
&\leq \frac \theta 2 \int_{B_1} \zeta^2 |\nabla v|^2 \dd x + 4\Theta\left(1 +\frac 1 {2\theta}\right) \int_{B_1} v^2 |\nabla \zeta|^2  \dd x  + \frac{r^2} \delta \left(3\int_{B_1} v^2 \dd x + \xi^2 \right) +  \frac {2\kappa^2 r^2} \eps \mathcal L^d(B_1) .
\end{split}
  \]
Combining terms, and choosing $\zeta$ radially decreasing in $B_1$ such that $\zeta = 1$ in $B_1$ and $\zeta = 0$ outside $B_1$, with $|\nabla \zeta|\leq 4$, we have
\[ 
\int_{B_1} |\nabla v|^2 \dd x \leq C\left(\left(\frac{r^2} \delta  + 1 \right) \int_{B_1}v^2 \dd x + \frac{\xi^2 r^2} \delta + \frac {\kappa^2 r^2} \eps \right),
\]
which implies the statement of the lemma.
\end{proof}

The following key lemma says that up to suitable rescalings, $u$ is locally approximated by $A_0$-harmonic functions, for some $A_0$ in the same ellipticity class as $A$.

\begin{lemma}\label{l:M}
Given $\tau>0$, there exists $M= M(d,\theta, \Theta,\delta,\eps, \tau)>0$ sufficiently small, such that for any rescaling $v$ of $u$, defined as in Lemma \ref{l:rescale}, such that 
\[
r^2(1+\kappa^2 + \xi^2) \leq M,
\]
and 
\[
\fint_{B_1} v^2 \dd x \leq 1,
\] 
there holds 
\begin{equation}\label{e:v-h-tau}
 \fint_{B_{1/2}} |v - h|^2 \dd x \leq \tau,
 \end{equation}
where $h$ is a weak solution of 
\begin{equation}\label{e:h-eqn}
 -\nabla\cdot ( A_0(x)\nabla h) = 0, \quad \text{in } B_{1/2},\\
\end{equation}
for some $A_0 \in M_{\theta,\Theta}(B_{1/2})$, and $h$ satisfies $\fint_{B_{1/2}} h^2 \dd x \leq 2^{d+2}$. In fact, we can take $h$ to be the $A_0$-harmonic lifting of $v$ in $B_{1/2}$. 
\end{lemma}

By a weak solution of \eqref{e:h-eqn}, we mean that $\int_{B_{1/2}} \nabla \psi \cdot (A_0(x) \nabla h)\dd x = 0$ for any $\psi\in H_0^1(B_{1/2})$.

\begin{proof}
%
Let $v(x) = \kappa u(x_0+rx)$ be as in the statement of the lemma, and define $A_0(x) = A(x_0+rx)$.  

We will need the following bound on $v$ in $H^1(B_{1/2})$, which follows from the Caccioppoli estimate of Lemma \ref{l:caccioppoli}: 
\begin{equation}\label{e:cacc-ineq}
\int_{B_{1/2}} |\nabla v|^2 \dd x \leq C\left( (\kappa^2 + \xi^2)r^2 + (1+r^2) \int_{B_1} v^2 \dd x\right) \leq C(1+ M) \left( 1 + \int_{B_1} v^2 \dd x\right).
\end{equation}


Next, define $h:B_{1/2} \to \R$ as the minimizer of
\[
\widetilde F(w) := \int_{B_{1/2}} \nabla w\cdot (A_0(x) \nabla w) \dd x,
\]
over
\[
\{ w \in H^1(B_{1/2}) : w- v \in H_0^1(B_{1/2}) \}.
\]
Let us also define $h(x)=v(x)$ for $x\in B_1\setminus B_{1/2}$.  As usual, this $h$ is the weak solution to
\[
\begin{cases}
 -\nabla \cdot (A_0(x) \nabla h) = 0, \quad x\in B_{1/2},\\
h - v \in H_0^1(B_{1/2}),
\end{cases}
\]
which in particular implies the identity $\int_{B_{1/2}} \nabla (v-h) \cdot (A_0(x) \nabla h) \dd x = 0$. Using this identity, Poincar\'e's inequality, and the minimizing property of $v$ (from Lemma \ref{l:rescale}), we have
\begin{equation}\label{e:chain}
\begin{split}
\fint_{B_{1/2}} (v-h)^2 \dd x &\leq C \int_{B_{1/2}} |\nabla (v-h)|^2 \dd x\\
&\leq \frac C \lambda \int_{B_{1/2}} \nabla (v-h) \cdot (A_0(x)\nabla (v-h)) \dd x\\
&= \frac C \lambda \left(\int_{B_{1/2}} \nabla v\cdot (A_0(x) \nabla v) \dd x - \int_{B_{1/2}} \nabla h\cdot (A_0(x) \nabla h) \dd x\right)\\
&\leq \frac {r^2} \delta \left( \left| \int_{B_1} (h+\xi)^2 \dd x - \nu\right| - \left| \int_{B_1} (v+\xi)^2 \dd x - \nu\right|\right)\\
&\quad + \frac{ r^2 \kappa^2} \eps \left( \left(\mathcal L^d(\{h \neq -\xi\}) - \gamma\right)_+ - \left(\mathcal L^d(\{v\neq -\xi\}) - \gamma\right)_+\right)\\
&\leq \frac {r^2} \delta \left| \int_{B_{1/2}} (h^2 - v^2 +2\xi(h-v))\dd x\right| + \frac {r^2\kappa^2}\eps \left( \mathcal L^d(\{h\neq - \xi\}) - \mathcal L^d(\{v\neq -\xi\})\right)_+,
\end{split}
\end{equation}
by the triangle inequality. (Here, $\nu$ and $\gamma$ are defined as in Lemma \ref{l:rescale}.) To bound the last expression, we begin with the second term:
\[
\frac{r^2 \kappa^2} \eps  \left( \mathcal L^d(\{h\neq - \xi\}) - \mathcal L^d(\{v\neq -\xi\})\right)_+ \leq \frac {2r^2\kappa^2} \eps \mathcal L^d(B_1) \leq C_1M,
\]
for some constant $C_1>0$ depending only on $d$ and $\eps$. Next, we have
\begin{equation}\label{e:term1}
\frac {r^2} \delta \left| \int_{B_{1/2}} (h^2 - v^2 +2\xi(h-v))\dd x\right| \leq \frac {C r^2} \delta \left( \int_{B_{1/2}} h^2 \dd x + \int_{B_{1/2}} v^2 \dd x + \xi^2\right).
\end{equation}
We need to estimate $\int_{B_{1/2}} h^2\dd x$, and this will be accomplished via a bound for  $\int_{B_{1/2}} (v-h)^2\dd x$ that is more crude than \eqref{e:chain}. Another application of the Poincar\'e inequality gives
\[
\begin{split}
\int_{B_{1/2}} (v-h)^2 \dd x &\leq C \int_{B_{1/2}} |\nabla (v-h)|^2 \dd x\\
&\leq C  \int_{B_{1/2}} (|\nabla v|^2 + |\nabla h|^2) \dd x\\
&\leq \frac C \lambda \left(\int_{B_{1/2}} \nabla v\cdot (A_0(x) \nabla v) \dd x +  \int_{B_{1/2}} \nabla h\cdot (A_0(x) \nabla h) \dd x\right)\\
&\leq \frac {2C} \lambda \int_{B_{1/2}} \nabla v\cdot (A_0(x) \nabla v) \dd x \\
&\leq \frac {2C\Lambda} \lambda \int_{B_{1/2}} |\nabla v|^2 \dd x \leq \frac {2C\Lambda} \lambda (1+M) \left(1 + \int_{B_1} v^2 \dd x\right).
\end{split}
\]
by the minimizing property of $h$ and by \eqref{e:cacc-ineq}. This implies 
\[
\begin{split}
\int_{B_{1/2}} h^2 \dd x &\leq 2\int_{B_{1/2}} v^2 \dd x + 2 \int_{B_{1/2}} (v-h)^2 \dd x\\
&\leq 2\int_{B_{1/2}} v^2 \dd x + \frac{4C\Lambda} \lambda (1+M) \left(1 + \int_{B_1} v^2 \dd x\right)\\
&\leq C_{d,\lambda,\Lambda}(1+M),
\end{split}
\]
since $\fint_{B_1} v^2 \dd x \leq 1$ by assumption. Taking $M<1$, we have shown $\int_{B_{1/2}} h^2 \dd x \leq C_2$ for some constant $C_2$ depending only on $\lambda$, $\Lambda$, and $d$. Returning to \eqref{e:term1}, we have
\[
\frac {r^2} \delta \left| \int_{B_{1/2}} (h^2 - v^2 +2\xi(h-v))\dd x\right| \leq \frac {Cr^2}\delta \left( C_2 + 1 + \xi^2\right) \leq C(1+C_2) M,
\]
by our assumption on $r$ and $\xi$. Combining all our estimates, we finally have
\[
\fint_{B_{1/2}} |v-h|^2 \dd x \leq C M,
\]
for a constant $C>0$ depending on $\delta$, $\eps$, $d$, $\lambda$, and $\Lambda$. Choosing $M< \tau/C$, we have established \eqref{e:v-h-tau}. We also choose $M$ small enough that 
\[
\fint_{B_{1/2}} h^2 \dd x \leq 2\fint_{B_{1/2}}(v-h)^2\dd x + 2\fint_{B_{1/2}} v^2\dd x \leq 2 CM +2^{d+1} \leq 2^{d+2},
\] 
as desired.
\end{proof}

The next lemma is a local oscillation estimate that uses the approximation of Lemma \ref{l:M}:

\begin{lemma}\label{l:local-holder}
There exist $M>0$, $\alpha\in (0,1)$, $\rho_0\in (0,1/4)$, and $K>0$, such that for any $x_0$, $\kappa$, $r$, and $\xi$ with $r^2(1+\kappa^2 + \xi^2) \leq M$ such that 
\[
v(x) = \kappa u(x_0+rx) - \xi, \quad x\in B_1,
\]
satisfies $\fint_{B_1} v^2 \dd x \leq 1$,  there holds
\begin{equation}\label{e:mu-bound}
 \fint_{B_{\rho_0}} |v - \mu|^2 \dd x \leq \rho_0^{2\alpha},
 \end{equation}
for some constant $\mu$ with $|\mu|\leq K$. 

The constants $\alpha$ and $K$ depend on $d$, $\theta$, and $\Theta$. The constants $\rho_0$ and $M$ depend on $d$, $\theta$, $\Theta$, $\delta$, and $\eps$.
\end{lemma}
 We intend to apply this lemma first to localize around a point $x_0\in B$ and ensure the rescaled $L^2$ norm is less than 1 by a suitable choice of $\kappa$, and then iteratively with $r$ and $\kappa$ chosen according to the scaling of  $\alpha$-H\"older continuity.
\begin{proof}
Let $\tau>0$ be a constant to be chosen later. With $v$ as in the statement of the lemma, let $h:B \to \R$ be the solution to $-\nabla \cdot (A\nabla h) =0$ with
\[ 
\fint_{B_{1/2}} |v - h|^2 \dd x < \tau,
\]
whose existence is guaranteed by Lemma \ref{l:M}. The choice of $\tau$ will determine $M$. Since $\fint_{B_{1/2}} h^2 \dd x \leq 2^{d+2}$, the interior regularity estimate for the equation satisfied by $h$ \cite[Theorem 8.24]{gilbargtrudinger} implies
\[  |
h(x) - h(0)|  \leq C |x|^\beta,
\]
for any $x\in B_{1/4}$, where $C>0$ and $\beta \in (0,1)$ depend on $d$, $\theta$, and $\Theta$. 
With $\rho_0 \in (0,1/4)$ to be chosen later, we now have
\begin{align*}
\fint_{B_{\rho_0}} |v(x) - h(0)|^2 \dd x &\leq 2\left( \fint_{B_{\rho_0}} |v(x) - h(x)|^2 \dd x + \fint_{B_{\rho_0}} |h(x) - h(0)|^2 \dd x\right)\\
 &\leq  2^{1-d} \rho_0^{-d}\tau + 2C \rho_0^{2\beta}.
\end{align*}
Choosing
\[ 
\rho_0 = \left(\frac 1 {4C} \right)^{1/\beta}, \quad \tau = 2^{d-2} \rho_0^{d+\beta},
\]
we have
\[
\fint_{B_{\rho_0}} |v(x) - h(0)|^2 \dd x \leq  \rho_0^{\beta}.
\]
Letting $\alpha = \beta/2$ and $\mu = h(0)$, we have shown \eqref{e:mu-bound}. The bound $|\mu|\leq K$ is a result of the interior $L^2$-to-$L^\infty$ estimate satisfied by $h$ \cite[Theorem 8.17]{gilbargtrudinger}.
\end{proof}

Next, we iterate Lemma \ref{l:local-holder} to show that under the same hypotheses, $v$ is H\"older continuous at the origin:

\begin{lemma}\label{l:rescaled-est}
Let $M$, $\alpha$, and $\rho_0$ be the constants granted by Lemma \ref{l:local-holder}, and let 
\[
v(x) = \kappa u(x_0+rx),
\]
be a rescaling of $u$ with $r^2(1+\kappa^2) \leq M$, such that $\fint_{B_1} v^2 \dd x \leq 1$. 
 Then there exists $C>0$ depending only on $d$, $\theta$, $\Theta$, $\delta$, and $\eps$, such that
\[
|v(x) - v(0)|\leq C|x|^\alpha, \quad \text{ if } |x|< \rho_0.
\]
\end{lemma}
\begin{proof}
Our goal is to show by induction that 
\begin{equation}\label{e:induction}
 \fint_{B_{\rho_0^j}} (v - \mu_j)^2 \dd x \leq \rho_0^{2j\alpha}, \quad j = 1,2, \ldots,
\end{equation}
for some convergent sequence $\mu_j$. The base case $j=1$ follows directly from Lemma \ref{l:local-holder}, with $|\mu_1|\leq K$.

Assume by induction that \eqref{e:induction} holds for some $j\geq 1$, with $|\mu_j|\leq K(1-\rho_0^{j\alpha})/(1-\rho_0^\alpha)$.  Note that the upper bound for $\mu_j$ is the $(j-1)$th partial sum of the geometric series $\sum_{i\geq 0} K \rho_0^{i\alpha}$. 
Define
\[ 
\begin{split}
v_j(x) &=  \frac {v(\rho_0^j x) - \mu_j}{\rho_0^{j\alpha}}\\
&= \frac{\kappa u(x_0+\rho_0^j rx) - \mu_j}{\rho_0^{j\alpha}},  \quad x\in B_1.
\end{split}
\] 
The inductive hypothesis \eqref{e:induction} implies 
\[
\fint_{B_1}v_j^2 \dd x = \rho_0^{-2j\alpha} \fint_{B_{\rho_0^j}} (v-\mu_j)^2 \dd x \leq 1.
\]
Since $v_j$ is a rescaling of $u$ with parameters $x_0$, $\kappa_j = \kappa \rho_0^{-j\alpha}$, $\xi_j = \mu_j \rho_0^{-j\alpha}$, and $r_j = r\rho_0^j$, we also need to check that $r_j^2(1+\kappa_j^2 + \xi_j^2)\leq M$. Indeed, since the inductive hypothesis implies $|\mu_j|\leq K(1-\rho_0^\alpha)^{-1}$, we have
\[ 
\begin{split}
r_j^2(1+\kappa_j^2 + \xi_j^2) &= r^2 \rho_0^{2j} (1+\kappa^2 \rho_0^{-2j\alpha} +  \mu_j^2 \rho_0^{-2j \alpha})\\
&\leq r^2 \rho_0^{2j} + \kappa^2 r^2 \rho_0^{2j(1-\alpha)} +  r^2 K^2 (1-\rho_0^\alpha)^{-2} \rho_0^{2j(1-\alpha)},
\end{split}
\]
Since $\rho_0\leq 1$ and $r^2(1+\kappa^2+\xi^2)\leq M$, we can choose $\rho_0$ smaller if necessary (depending on $K$, which depends only on $d$, $\theta$, and $\Theta$) so that $\rho_0^{j(2-\alpha)} r^2K^2(1-\rho_0^\alpha)^{-2}\leq \frac 1 2 M$ and $r^2\rho_0^{2j} + \kappa^2 r^2\rho_0^{2j(1-\alpha)} \leq \frac 1 2 r^2(1+\kappa^2)$ for any $j$, and as a result, $r_j^2(1+\kappa_j^2 + \xi_j^2)\leq M$. 

Since all the hypotheses of Lemma \ref{l:local-holder} are satisfied, we obtain
\[ 
\fint_{B_{\rho_0}} |v_j - \mu|^2 \dd x \leq  \rho_0^{2\alpha},
\]
for some constant $\mu$ with $|\mu|\leq K$. Translating back to $v$ with the change of variables $x \mapsto \rho_0^j x$, this bound becomes
\[ 
\fint_{B_{\rho_0^{j+1}}} |v - \mu_j - \rho_0^{j\alpha} \mu|^2 \dd x \leq  \rho_0^{2\alpha(j+1)},
\]
and we have shown \eqref{e:induction} with $\mu_{j+1} = \mu_j + \rho_0^{j\alpha}\mu$. Using our upper bound for $\mu_j$, we have
\[
|\mu_{j+1}| \leq |\mu_j| + \rho_0^{j\alpha}|\mu| \leq K \frac {1-\rho_0^{j\alpha}} {1-\rho_0^\alpha} + K \rho_0^{j\alpha} = K\frac{1-\rho_0^{(j+1)\alpha}} {1-\rho_0^\alpha},
\]
which allows us to close the induction.

Next, we claim $\mu_j$ is a convergent sequence. Indeed, our work above shows $|\mu_{j+1} - \mu_j| \leq \rho_0^{j\alpha} K$ for all $j\geq 1$, which implies
\[
|\mu_{j+k} - \mu_j| \leq  \left(\rho_0^{j\alpha} + \rho_0^{(j+1)\alpha} + \cdots + \rho_0^{(j+k-1)\alpha} \right)K < \frac {\rho_0^{j\alpha} }{1-\rho_0^{\alpha}} K,
\]
for any $k\geq 1$, and therefore $\mu_j$ is a Cauchy sequence. Letting $\mu_0 = \lim_{j\to\infty} \mu_j$, we let $k\to\infty$ in our estimate of $|\mu_{j+k} - \mu_j|$ to obtain $|\mu_0 - \mu_j| \leq \rho_0^{j\alpha}(1-\rho_0^\alpha)^{-1}K$.

Finally, for any $\rho\in (0,\rho_0)$, choose $j>0$ such that $\rho_0^{j+1} < \rho \leq \rho_0^j$. From \eqref{e:induction}, we have
\[ 
\fint_{B_\rho} |v-\mu_0|^2 \dd x \leq  \frac{2 \rho_0^{dj}} {\rho^d}\fint_{B_{\rho_0^j}} |v-\mu_j|^2 \dd x + 2 |\mu_j - \mu_0|^2  \leq 2\left( \frac 1 {\rho_0^d}  + \frac { K^2 }{(1-\rho_0^\alpha)^2}\right)\rho_0^{2j\alpha}\leq C \rho^{2\alpha},
\]
with $C$ depending only on $d$, $\theta$, $\Theta$, $\delta$, and $\eps$. This implies the conclusion of the lemma.
\end{proof}

We are ready to prove our main H\"older estimate:

\begin{theorem}\label{t:holder}
Any minimizer $u$ of $F_{\delta,\eps}$ over $H_0^1(B)$ is locally H\"older continuous in $B$, and for any compact $K\subset B$, there holds 
\[
\|u\|_{C^\alpha(K\cap B)} \leq C \|u\|_{L^2(B)}.
 \]
 The constant $\alpha\in (0,1)$ depends on $d$, $\theta$, and $\Theta$. The constant $C>0$ depends on $d$, $\theta$, $\Theta$, $\delta$, and $\eps$.
\end{theorem}

\begin{proof} 
Let $x_0 \in B$ be fixed. To recenter near $x_0$, we define
\[ 
v(x) = \kappa u(x_0+rx),
\]
with $\kappa, r>0$ chosen so that the hypotheses of Lemma \ref{l:rescaled-est} are satisifed. In particular, letting $M$ be the constant granted by Lemma \ref{l:local-holder}, we take
\[
r:= \min\left(\frac {\mbox{dist}(x_0,\partial B)} 2 , \sqrt{\frac {M} 2}\right),
\]
and 
\[
\kappa :=  \min\left( 1, \sqrt{\frac {1} {\fint_{B_r(x_0)}u^2 \dd x}}\right).
\]
We then have $r^2(1+\kappa^2) \leq M$ and $\fint_{B_1} v^2 \dd x \leq 1$, and Lemma \ref{l:rescaled-est} implies $|v(x) - v(0)| \leq C|x|^\alpha$ whenever $|x|< \rho_0$, where $\alpha$ and $\rho_0$ are the constants granted by Lemma \ref{l:local-holder}, and $C$ depends only on $d$, $\theta$, $\Theta$, $\delta$, and $\eps$. 

Since $x_0\in B$ was arbitrary and $B$ is a bounded set, the proof is complete after translating from $v$ back to $u$.
\end{proof}

%
%
%
%

Now we can prove Theorem \ref{t:main}, our main result. Choosing fixed values of $\delta$ and $\eps$ that are small enough as in Lemma \ref{l:constraint}, and letting $u_*$ be a nonnegative minimizer of $F_{\delta,\eps}$ over $H_{0}^1(B)$ and $\Omega_* = \{u_*>0\}$, we have $\int_B u_*^2 \dd x = 1$ and $\mathcal L^d(\Omega_*) \leq 1$. Theorem \ref{t:holder} implies $u_*$ is continuous in $B$, and therefore $\Omega_*$ is open. The H\"older estimate for $u_*$ depends on $d$, $\theta$, $\Theta$, the distance to $\partial B$, and (via $\eps$) $\mathcal L^d(B)$. From Proposition \ref{p:equiv}, $\Omega_*$ is a minimizer of $\lambda_1^A$ over $\mathcal C'$.

We could have $\mathcal L^d(\Omega_*) < 1$, i.e. $\Omega_*\in \mathcal C' \setminus \mathcal C$. In this case, we define for $r>0$ the open set
\[
\Omega_r = \bigcup_{x\in \Omega_*} B_r(x)\cap B.
\]
By continuity, there is some $r>0$ with $\mathcal L^d(\Omega_r) = 1$, and since $\lambda_1^A$ is nonincreasing with respect to set inclusion, $\Omega_r$ is a minimizer of $\lambda_1^A$ over the class $\mathcal C$. The eigenfunction $u_r$ on $\Omega_r$ is clearly a minimizer of $F_{\delta,\eps}$, and therefore the H\"older estimate of Theorem \ref{t:holder} applies to the optimal eigenfunction in this case as well.

\section{Linear growth at boundary points}\label{s:linear}

This section is devoted to the proof of Theorem \ref{t:boundary}. We choose $\delta$ and $\eps$ small enough that we can apply Proposition \ref{p:equiv}. (These values of $\delta$ and $\eps$ will remain fixed throughout this section.) We let $u= u_*$ be a minimizer of $F_{\delta,\eps}$. Then $u$ is a solution of the eigenvalue equation 
\begin{equation}\label{e:eigeqn}
-\nabla \cdot (A\nabla u) = \lambda_* u
\end{equation}
 in $\Omega_* = \{u>0\}$. By Theorem \ref{t:main}, $\Omega_*$ is open.

The proof in this section is based on rescalings of $u$ that are centered around boundary points with $u(x_0) = 0$. The first step is to revisit the approximation argument of Lemma \ref{l:M} armed with the H\"older estimate of Lemma \ref{l:rescaled-est}:

\begin{lemma}\label{l:Nzero}
Given $\tau>0$, there exists $N = N(d,\theta,\Theta,\lambda_*,\tau)>0$ small enough, such that for any $x_0\in B$ with $u(x_0) = 0$, and any rescaling
\[
v(x) = \kappa u(x_0+rx), \quad x\in B_1,
\]
of $u$, with $\kappa>0$ and $0< r < \dist(x_0, \partial B)$ such that 
\[
r^2(1+\kappa^2) < N\]
and $\fint_{B_1} v^2 \dd x \leq 1$, 
there holds
\[
\sup_{B_{1/8}} v \leq \tau.
\]
\end{lemma}

\begin{proof}
Since $u$ satisfies \eqref{e:eigeqn}, a direct calculation shows that $v$ satisfies
\begin{equation}\label{e:rescaled-eig}
-\nabla \cdot (A_0(x) \nabla v) = r^2 \lambda_*v, \quad x\in B_1\cap \{v>0\},
\end{equation}
where $A_0(x) = A(x_0+rx)$. Next, let $h$ be the $A_0$-harmonic lifting of $v$ in $B_{1/2}$, i.e. the weak solution of
\begin{equation}\label{e:h-eqn}
\begin{split}
\nabla \cdot (A_0(x) \nabla h) &= 0, \quad x\in B_{1/2},\\
h - u  &\in H_0^1(B_{1/2}).
\end{split}
\end{equation}
Let us choose $N$ small enough so that, with $r^2(1+\kappa^2)< N$, (i) the H\"older estimate of Lemma \ref{l:rescaled-est} applies, and (ii) from Lemma \ref{l:M} with $\xi=0$, the inequality 
\begin{equation}\label{e:tau-prime}
\fint_{B_{1/2}}|v - h|^2 \dd x \leq \tau'
\end{equation}
holds, for some $\tau'$ to be determined below. 

The strategy of the proof is as follows: first, show $h$ is small near $x=0$, and secondly, leverage \eqref{e:tau-prime} to conclude $v$ is also small near $x=0$. 

To establish the smallness of $h$, we apply the Harnack inequality \cite[Theorem 8.20]{gilbargtrudinger} followed by an $L^2$-to-$L^\infty$ estimate \cite[Theorem 8.17]{gilbargtrudinger} to obtain, for any $\rho\in (0,\frac 1 2)$,
\begin{equation}\label{e:sup1}
\sup_{B_{1/4}} h \leq C_1 h(0) \leq C_2 \rho^{-d/2} \|h\|_{L^2(B_\rho)},
\end{equation}
where $C_1$ and $C_2$ depend only on $d$, $\theta$, and $\Theta$. Next, proceeding in a similar manner to the proof of Lemma \ref{l:local-holder}, we have
\begin{equation}\label{e:sup2}
\begin{split}
\rho^{-d} \|h\|_{L^2(B_\rho)}^2 &= \omega_d \fint_{B_\rho} h^2 \dd x \\
&\leq 2\omega_d\left( \fint_{B_\rho} |v-h|^2 \dd x + \fint_{B_\rho} v^2 \dd x\right)\\
&\leq 2^{1-d} \rho^{-d} \omega_d \tau' + 2\omega_d \fint_{B_\rho} v^2 \dd x,
\end{split}
\end{equation}
by \eqref{e:tau-prime}. Next, we bound the last term on the right in \eqref{e:sup2} by using the H\"older estimate of Lemma \ref{l:rescaled-est} along with the fact that $v(0) = u(x_0) = 0$ to obtain
\[
\fint_{B_\rho} v^2 \dd x = \fint_{B_\rho} |v(x) - v(0)|^2 \dd x \leq C^2 \rho^{2\alpha}, 
\]
if $\rho< \rho_0$, where $C$, $\alpha$, and $\rho_0$ are the constants from Lemma \ref{l:rescaled-est}. We conclude
\begin{equation}\label{e:h-bound}
\sup_{B_{1/4}} h \leq C_2 \left( 2^{1-d} \rho^{-d} \omega_d \tau' + 2\omega_d C^2 \rho^{2\alpha}\right)^{1/2}.
\end{equation}

To bound $\sup_{B_{1/4}} v$, we need to apply the local $L^2$-to-$L^\infty$ estimate at the boundary, provided by \cite[Theorem 8.25]{gilbargtrudinger}:
\[
\sup_{B_{1/8}\cap\{v>0\}} v \leq C_3 \|v\|_{L^2(B_{1/4})}.
\]
 It is important to note that this estimate requires no smoothness for the boundary of $\{v>0\}$, and the constant $C_3$ depends only on $d$, $\theta$, $\Theta$, and $\lambda_*$ (since $v$ satisfies the elliptic equation \eqref{e:rescaled-eig} with zeroth-order coefficient $r^2 \lambda_*$, and we can ensure $r\leq 1$ by choosing $N\leq 1$.).

We then have
\[
\begin{split}
\sup_{B_{1/8}\cap \{v>0\}} v &\leq C_3 \left(\|v-h\|_{L^2(B_{1/4})} + \|h\|_{L^2(B_{1/4})}\right)\\
&\leq C_3 2^{d/2} \sqrt{\tau'} + C_3 \omega_d 4^{-d} \left(2^{1-d} \rho^{-d} \omega_d \tau' + 2 \omega_d C^2 \rho^{2\alpha}\right)^{1/2}
\end{split}
\]
by \eqref{e:tau-prime} and \eqref{e:h-bound}. Choosing $\rho< \rho_0$ small enough depending on $\tau$, $\alpha$, $C$, $C_3$, and $\omega_d$, and then choosing $\tau'$ depending on $\tau$, $\rho$, $d$, $C_3$, and $\omega_d$, we can ensure
\[
\sup_{B_{1/8}} v = \sup_{B_{1/8} \cap \{v>0\}} v \leq \tau.
\]
Tracing the dependence of all quantities, we see that $\tau'$ can be chosen depending only on $\tau$, $d$, $\theta$, $\Theta$, and $\lambda_*$. This $\tau'$ determines the choice of $N$ in the statement of the lemma, and the proof is complete.
\end{proof}

Finally, by applying Lemma \ref{l:Nzero} iteratively, we obtain the desired linear growth estimate:


\begin{proof}[Proof of Theorem \ref{t:boundary}]
We prove the growth estimate for any minimizer $u$ of $F_{\delta,\eps}$. The statement of the theorem then follows by choosing $\delta$ and $\eps$ small enough so that $u_*$ is a minimizer of $F_{\delta,\eps}$. 

Let $\tau = \frac 1 8$, and let $N$ be the corresponding constant granted by Lemma \ref{l:Nzero}. Given $x_0\in \partial\{u>0\}$, we recenter around $x_0$  in a similar fashion to the proof of Theorem \ref{t:holder}, by defining
\[ 
v = \kappa u(x_0+rx).
\]
where
\[
r:= \min\left(\frac {\mbox{dist}(x_0,\partial B)} 2 , \sqrt{\frac {N} 2}\right),
\]
and
 \[
\kappa :=  \min\left( 1, \sqrt{\frac {1} {\fint_{B_r(x_0)}u^2 \dd x}}\right).
\]
Then, $r^2(1+\kappa^2) < N$ and $\fint_{B_1} v^2 \dd x \leq 1$.

We would like to show by induction that 
\begin{equation}\label{e:induction2}
\sup_{B_{8^{-j}}} v \leq \frac 1 {8^j}, \quad j=1,2,\ldots,
\end{equation}
The base case follows from applying Lemma \ref{l:Nzero} directly to $v$. If \eqref{e:induction2} holds for some $j\geq 1$, then we define
\[
\begin{split}
v_j(x) &= 8^j v(8^{-j} x)\\
&= 8^j \kappa u(x_0+8^{-j}rx), \quad x\in B_1.
\end{split}
\]
From the inductive hypothesis, we have 
\[
\fint_{B_1} v_j^2 \dd x = 8^{2j} \fint_{B_{8^{-j}}} v^2 \dd x \leq 1.
\]
We also clearly have $v_j(0) = 0$. With $\kappa_j = 8^j\kappa$ and $r_j = 8^{-j} r$, we have $r_j^2(1+\kappa_j^2) = r^2(4^{-2j} + \kappa^2) < N$. Therefore, Lemma \ref{l:Nzero} implies 
\[
\sup_{B_{1/8}} v_j \leq \frac 1 8,
\]
or 
\[
\sup_{B_{8^{-j-1}}} v \leq 8^{-j-1},
\]
and we have established \eqref{e:induction2}.

Next, for any $x\in B_{1/8}$, there is a $j\geq 1$ such that $8^{-j-1} < |x| \leq 8^{-j}$. From \eqref{e:induction2}, we have
\[
|v(x)|\leq \sup_{B_{8^{-j}}} w \leq 8^{-j} \leq 8 |x|,
\]
and the proof is complete after translating from $v$ back to $u$.
\end{proof}

\bibliographystyle{abbrv}
\bibliography{eigenvalue}

\end{document}